\documentclass{amsart}
\usepackage{amssymb,amsthm}
\usepackage{amsfonts}
\usepackage{enumerate}
\usepackage{amsmath}
\usepackage{amsrefs}
\usepackage{graphicx}
\usepackage[utf8]{inputenc}
\usepackage[english]{babel}
\usepackage{color}
\usepackage{multicol}
\usepackage{xcolor}
\definecolor{bluegreen2}{RGB}{76, 134, 180,}
\definecolor{green3}{RGB}{1, 25, 16} 
\definecolor{ashgrey}{RGB}{47,79,79}

\usepackage[linktocpage]{hyperref}
\hypersetup{colorlinks=true, linkcolor=bluegreen2, citecolor=green3, filecolor=magenta, urlcolor=bluegreen2}\usepackage{tikz, tikz-cd} \usepackage{stmaryrd}

\theoremstyle{plain}

\newtheorem{corollary}{Corollary}[section]

\newtheorem{definition}{Definition}[section]
\newtheorem{example}{Example}[section]

\newtheorem{proposition}{Proposition}[section]
\newtheorem{remark}{Remark}[section]

\newtheorem{theorem}{Theorem}[section]
\numberwithin{equation}{section}

\definecolor{lightgrey}{cmyk}{0,0,0,0.30}
\definecolor{darkgrey}{cmyk}{0,0,0,0.70}
\definecolor{purple}{cmyk}{0.45,0.86,0,0}
\definecolor{darkblue}{cmyk}{1.7,0.7,0.3,0}
\definecolor{lightblue}{cmyk}{0.6,0.2,0.1,0}
\definecolor{midblue}{cmyk}{1.1,0.3,0.2,0}

\begin{document}	
	\title[A new class of $p$-adic Lipschitz functions and Hensel's lemma]{A new class of $p$-adic Lipschitz functions and multidimensional Hensel's lemma}
	
	\author[F. Bolivar-Barbosa]{Fausto Bolivar-Barbosa}
	\address{Departamento de Matem\'aticas.
		Universidad Nacional de Colombia. Ciudad Universitaria. Carrera 30 Calle 45. Edificio 404. Bogot\'a, Colombia.
	} 	
	\email{fcbolivarb@unal.edu.co}
	
	\author[E. Le\'{o}n-Cardenal]{Edwin Le\'{o}n-Cardenal}
	\address{CONACYT -- Centro de Investigaci\'{o}n en Matem\'{a}ticas, 
		Unidad Zacatecas.
		Quantum Ciudad del Conocimiento.
		Avenida Lasec, Andador Galileo Galilei,
		Manzana 3 Lote 7. C.P. 98160.
		Zacatecas, ZAC.
		M\'{e}xico.}
	\email{edwin.leon@cimat.mx}
	\thanks{E. León Cardenal is partially supported by CONACYT Grant No. 286445.}
	
	\author[J.J. Rodr\'iguez-Vega]{John Jaime Rodr\'iguez-Vega}
	\address{Departamento de Matem\'aticas.
		Universidad Nacional de Colombia. Ciudad Universitaria. Carrera 30 Calle 45. Edificio 404. Bogot\'a, Colombia.
	}	
	\email{jjrodriguezv@unal.edu.co}
	\date{}
	
	

	\begin{abstract}
		In this work we study $p$-adic continuous functions in several variables taking values on $\mathbb{Z}_p$. We describe the orthonormal van der Put base of these functions and study various Lipschitz conditions in several variables, generalizing previous work of Anashin. In particular, we introduce  a new class of $p$-adic Lipschitz functions and study some of their properties. We also prove a  Hensel's lifting lemma for this new class of functions, generalizing previous results of Yurova and Khrennikov.
	\end{abstract}
	\maketitle
	\section{Introduction}
	The theory of non-Archimedean dynamical systems has been studied intensively in recent years, not only from the theoretical point of view but also from a practical one. Examples of non-Archimedean local fields include $\mathbb{Q}_p$, the field of $p$-adic numbers (see Section \ref{Sec2} for a formal definition), and $\mathbb{F}_p((T))$, the field of formal Laurent series with coefficients in the finite field with $p$ elements. There are at least two big courses in the theory of non-Archimedean dynamics. The first one is done over $\mathbb{C}_p$, which is the complete algebraic closure of  $\mathbb{Q}_p$, this wing of the theory also includes dynamics over Berkovich spaces, see the recent book \cite{Be}. The second approach is about the dynamical ergodic theory of $\mathbb{Q}_p$ (or of a finite extension of it), see the also recent book \cite{An-Kh_Book}. In this article we will be concerned with the latter approach. 
	
	Some of the problems in the dynamical ergodic theory of $\mathbb{Q}_p$ can be addressed by understanding the  $\mathbb{Q}_p$-Banach space of continuous functions from $\mathbb{S}$ to $\mathbb{Q}_p$, for $\mathbb{S}\subseteq\mathbb{Q}_p$. Under some assumptions on $\mathbb{S}$, it is possible to define a notion of orthogonality for continuous functions from $\mathbb{S}$ to $\mathbb{Q}_p$. Then we have several choices for `orthogonal' or `orthonormal' bases, being the Mahler and the van der Put bases the main `orthonormal' ones. One may then try to characterize some ergodic properties of continuous functions in terms of the coefficients of a given basis. This approach has proven adequate for the class of $p$-adic Lipschitz functions, for instance it is known the description of the $1$-Lipschitz functions which are measure preserving in terms of the van der Put basis, see  \cite{Yu-Kh_JNT13}. Other characterizations of measure preserving, ergodic and locally scaling functions, are given in e.g. \cites{SchiBook,An94,An02,An,An-Kh_Book,An-Kh-Yu,Yu-Kh_JNT,Je-Li,Je_JNT,Je_PNUAA,Fur,Me_IJNT,Me_Coll}. Some applications of this dynamical ergodic theory to computer science are given in e.g. \cites{An94,An02,An,An-Kh_Book}.
	
	Much less is known in the case of several variables, i.e. when $\mathbb{S}\subseteq\mathbb{Q}_p^n$. It seems that the study of $p$-adic ergodic dynamics in several variables is not equally developed as the one variable case. To the best of our knowledge, the only notable exceptions are the works of Anashin in \cites{An94,An02,An-Kh_Book}. This work attempts to motivate the further development of $p$-adic dynamics in several variables. In this note we take the first step is this direction by studying continuous functions and a class of $p$-adic Lipschitz functions in several $p$-adic variables. Anashin has given in \cites{An94,An02,An-Kh_Book} a straightforward generalization of the Lipschitz condition for the multivariable case, but in the world of several variables there are other possibilities and we report here a new class of $p$-adic Lipschitz functions in several variables depending on several parameters. This Lipschitz condition can be interpreted as a $p$-adic weighted Lipschitz condition with weight $\boldsymbol{\alpha}=(\alpha_1,\ldots,\alpha_n)\in\mathbb{Z}_{\geq0}^n$, see  Definition \ref{Def:New_p_alpha}. We then study the functional subspace of $p$-adic Lipschitz functions in several variables and provide some properties of the coefficients of the van der Put expansion of our new type of Lipschitz functions. As application of our findings we present a version in several variables of the
	Hensel's lifting lemma, proved by Yurova and Khrennikov in \cite{Yu-Kh_JNT}. 
	
	The work is organized as follows. In Section \ref{Sec2} we present the basics of the $p$-adic analysis of continuous functions from $\mathbb{Z}_p$ to $\mathbb{Q}_p$. With the purpose of later use we revisit the Hensel's lifting lemma in \cite{Yu-Kh_JNT}*{Thm 3.3} and provide an equivalent new statement in Theorem \ref{Thm:Hensel_alpha_Univ}. Section \ref{Sec3} contains the main results of this work. 
	We present in Theorem \ref{Thm:vdPut_multi} a description of the orthonormal van der Put base of the $\mathbb{Q}_p$-Banach space of continuous functions from $\mathbb{Z}_p^n$ to $\mathbb{Q}_p$. We then study $p$-adic Lipschitz functions in several variables, giving in Theorem \ref{Thm:alpha_lip_to_vdPut_multivar} some properties of their van der Put coefficients. Finally we prove in Theorem \ref{Thm:Main} a generalization of Hensel's lemma of \cite{Yu-Kh_JNT} for our class of $p$-adic Lipschitz functions in several variables.
	
	We believe that the theory of higher dimensional $p$-adic functions developed here may stimulate future studies on $p$-adic ergodic dynamics. For instance it would be interesting to find analogues of some of the results in \cites{An94,An02,An-Kh_Book} in terms of the van der Put base of Theorem \ref{Thm:vdPut_multi}. One may try also to use orthogonal bases to study other dynamical properties of functions in several variables, including the study of Bernoulli maps, or more generally of locally scaling functions like in \cites{Fur,Je_JNT,Me_Coll}. Finally it is very natural to try to extend the results of \cite{Ka-Sto} to higher dimensional $p$-adic functions.
	\section*{Acknowledgments}
	The authors want to thank to Ana Cecilia Garc\'ia-Lomel\'i for many suggestions and a careful reading of a previous version of this work.
	
	
	\section{Elements of univariate $p$-adic analysis}\label{Sec2}
	\subsection{$p$-adic numbers and $p$-adic functions}
	In this section we summarize some basic aspects of the field of $p$-adic
	numbers, for an in-depth discussion the reader may consult e.g. \cites{SchiBook,RoBook,An-Kh_Book}.
	
	We fix a prime number $p$. Let $x$ be a  non-zero rational number. Then, $x=p^{k}\frac{a}{b}$, with $p\nmid ab$, and $k\in\mathbb{Z}$. The \textit{$p$-adic absolute value} of $x$ is defined as	\[
	|x|_p=
	\begin{cases}
		p^{-k}, & \text{ if }x\neq0,\\
		0, & \text{ if }x=0.
	\end{cases}
	\]
	
	The \textit{$p$-adic distance} over $\mathbb{Q}$ is defined as 
	$d(x,y):=|x-y|_p,$ for $ x,y\in \mathbb{Q}$. The \textit{field of $p$-adic numbers} $\mathbb{Q}_{p}$ is defined as the completion of $\mathbb{Q}$ with respect to the distance $d$. Any
	$p$-adic number $x\neq0$ has a unique representation of the form
	\begin{equation}\label{eq:p_adic_expansion}
		x=p^{\gamma}\sum_{i=0}^{\infty}x_{i}p^{i}, 
	\end{equation}
	where $\gamma=\gamma(x)\in\mathbb{Z},\ x_{i}\in\{0,1,\dots,p-1\},\ x_{0}\neq0$. The integer $\gamma$ is called the \textit{ $p$-adic order of} $x$, and it will be denoted as $ord(x)$. By definition $ord(0)=+\infty$.
	
	A relevant fact about the $p$-adic norm $|\cdot|_p$ is that it is \textit{ultrametric} or \textit{non-Archimedean}, i.e. one has \[
	|x+y|_p\leq\max\{|x|_p,|y|_p\},\quad\text{ for any } x,y\in\mathbb{Q}_p. 
	\]
	
	A basis of open sets for the topology of the metric space $(\mathbb{Q}_p,d)$, is given by the \textit{open balls} $B_{r}(a)$ with
	center $a\in\mathbb{Q}_p$ and radius $p^{r}$ (with $r\in\mathbb{Z}$):
	\[
	B_{r}(a)=\{x\in \mathbb{Q}_p\ :\ |x-a|_p\leq p^r\}.
	\]
	
	The unit ball
	\[\mathbb{Z}_{p}=\{x\in\mathbb{Q}_{p}\ :\ |x|_{p}\leq1\}=\{x\in\mathbb{Q}_{p}\ :\ x=\sum_{i=i_{0}}^{\infty}x_{i}p^{i},i_{0}\geq0\},
	\]
	is a compact set in $(\mathbb{Q}_p,d)$. It is also a local ring with maximal ideal $p\mathbb{Z}_{p}$. The \textit{residue field }of
	$\mathbb{Q}_{p}$ is $\mathbb{Z}_{p}/p\mathbb{Z}_{p}\cong\mathbb{F}_{p}$,  the
	finite field with $p$ elements.
	
	From \eqref{eq:p_adic_expansion} it follows that any $x\in\mathbb{Q}_p$ is a limit of a sequence $\{x^{(n)}\}_{n\in\mathbb{N}}$ of rational numbers
	\[x^{(n)}=p^\gamma(x_0+x_1p+\cdots+x_np^n).\]
	The sequence $\{x^{(n)}\}_{n\in\mathbb{N}}$ is called the \textit{standard sequence}, see \cite{SchiBook}*{Sec. 62}. The standard sequence of an element $x\in\mathbb{Z}_p$ consists of
	non-negative integers and it is eventually constant if $x\in\mathbb{Z}$. For a non-negative integer $m$ and a $p$-adic integer $x$ we will write \[m\vartriangleleft x\] if $m$ is one of the numbers $x^{(0)},x^{(1)},\ldots$. In this case we will say that $m$ is \textit{an initial part of} $x$. The following definition will be used later on.
	\begin{definition}\label{def:m_star}
		If $m\in\mathbb{Z}_{\geq 0}$, then \eqref{eq:p_adic_expansion} takes the form $m=m_0+m_1p+\cdots+m_{s-1}p^{s-1}+m_sp^s$, with $m_s\neq 0$. In this case we take $s=s(m)=\lfloor \log_p m\rfloor$ and set for $s(m)\geq 1$ (equivalently $m\geq p$) \[m^\ast:=m_0+m_1p+\cdots+m_{s-1}p^{s-1}.\]
	\end{definition}
	
	\subsection{Continuous functions and van der Put Bases}
	The notion of continuity over $\mathbb{Q}_p$ is directly borrowed from the classical analysis over $\mathbb{R}$. 
	If $\mathbb{S}\subseteq\mathbb{Q}_p$ then a function $f:\ \mathbb{S} \to \mathbb{Q}_p$ is continuous at $s\in \mathbb{S}$ if for each $\epsilon > 0$ there is a $\delta > 0$ such that $|x - s|_p< \delta$, implies $|f(x)-f(s)|_p<\epsilon$. Something similar is true for the uniform convergence. 
	An important source of examples of continuous functions is given by the locally constant functions.
	\begin{definition}\label{Def:locContUniv}
		A function $f:\ \mathbb{S} \to \mathbb{Q}_p$  is locally constant if for each $s\in \mathbb{S}$ there exists a neighbourhood $U$ of $s$ such that $f$ is constant on $U\cap \mathbb{S}$.
	\end{definition}
	For instance the characteristic function of $\mathbb{Z}_p$ is a locally constant function. 
	
	We shall consider the $\mathbb{Q}_p$-vector space of continuous functions from  $\mathbb{S}\subseteq\mathbb{Q}_p$ to $\mathbb{Q}_p$,  denoted by $C(\mathbb{S}\to \mathbb{Q}_p)$. This set may be consider as a non-Archimedean Banach space by  endowing it with the supremum norm 
	\[||f||_\infty=\sup_{x\in \mathbb{S}}|f(x)|_p\quad\text{for}\quad f\in C(\mathbb{S}\to \mathbb{Q}_p).\]
	Then the set of locally constant functions on $\mathbb{S}$ form a $\mathbb{Q}_p$-linear subspace of $C(\mathbb{S}\to \mathbb{Q}_p)$ and moreover every continuous functions can be uniformly approximated by locally constant functions, see e.g. \cite{SchiBook}*{Thm. 26.2}. In addition, it is possible to define a notion of orthogonality over $C(\mathbb{S}\to \mathbb{Q}_p)$, see \cite{SchiBook}*{Sect. 50}, and one has that there exists an `orthonormal' basis of $C(\mathbb{Z}_p\to \mathbb{Q}_p)$ formed by locally constant functions, the van der Put Basis, which is described as follows. 
	\begin{theorem}[\protect{\cite{SchiBook}*{Thm 62.2}}]\label{Thm:vdP_base_univ}
		For $x\in\mathbb{Z}_p$ and $m\in\mathbb{Z}_{\geq 0}$, the functions $e_0,e_1,\ldots$ defined by 
		\[
		e_m(x)=
		\begin{cases}
			1, & \text{ if } m\vartriangleleft x,\\
			0, & \text{ otherwise},
		\end{cases}
		\]
		form an `orthonormal' basis (the van der Put basis) of the space $C(\mathbb{Z}_p\to \mathbb{Q}_p)$. If $f: \mathbb{Z}_p\to \mathbb{Q}_p$ is continuous and has the expansion 
		\begin{equation}\label{eq:vderp_univariate}
			f(x)=\sum_{m=0}^{\infty} B_me_m(x),\qquad x\in\mathbb{Z}_p,
		\end{equation}
		then for $m\in\mathbb{Z}_{\geq 0},\quad B_m=\begin{cases}
			f(m)-f(m^\ast), & \text{ if } m\geq p,\\
			f(m), & \text{ otherwise}.
		\end{cases}$
	\end{theorem}
	\begin{remark}
		There are several `orthogonal' basis of  $C(\mathbb{Z}_p\to \mathbb{Q}_p)$. Notably the Mahler basis is another important one, see the historical notes and comments in  \cites{RoBook,SchiBook}. 
	\end{remark}
	\subsection{Lipschitz conditions}
	The notion of differentiable functions over $\mathbb{Q}_p$ is also borrowed for the classical analysis over $\mathbb{R}$. However, it is well known that this classical notion is not very useful, see e.g \cite{RoBook}*{Ch. 5, Sec. 1} or \cite{SchiBook}*{Ch. 2, Sec. 26}. Consider for instance, the continuous function $f:\ \mathbb{Z}_p\to \mathbb{Z}_p$ defined by $f(\sum_{i=0}^{\infty}x_{i}p^{i})=\sum_{i=0}^{\infty}x_{i}p^{2i}$. It is not difficult to see that $f$ is differentiable at every point of the domain with derivative identically zero. But $f$ is also injective, so in particular is very far from being locally constant. 
	
	There are several approaches to overcome this feature of $p$-adic analysis, for instance one may play with the definition of differentiability in order to get meaningful properties of $p$-adic functions, see some of these strategies at  \cites{SchiBook,RoBook,An-Kh_Book}. Another path that one may follow is to study some relevant functional subspaces of $C(\mathbb{S}\to \mathbb{Q}_p)$ or subspaces of the set of analytic functions over $\mathbb{Q}_p$. We focus here on the subspace of $C(\mathbb{Z}_p\to \mathbb{Q}_p)$ formed by the functions which satisfy a Lipschitz type condition. 
	\begin{definition}
		Take $\alpha\in\mathbb{Z}_{\geq 0}$. A function $f: \mathbb{Z}_p\to \mathbb{Z}_p$ is called a $p^\alpha-$Lipschitz function if for every $x,y\in\mathbb{Z}_p$,
		\begin{equation}\label{Eq:p_alpha_Lipschitz}
			|f(x)-f(y)|_p\leq p^\alpha|x-y|_p.
		\end{equation}
		The set consisting of all $p^\alpha-$Lipschitz functions is denoted $\mathit{Lip}_{\alpha}$. The set of $1$-Lipschitz functions, i.e. when $\alpha=0$, is denoted $\mathit{Lip}_{0}$.
	\end{definition}
	\begin{example}
		Any polynomial function with coefficients in $\mathbb{Z}_p$ is a $1$-Lipschitz function. The function $f(x)=\frac{x-x^p}{p}$ is a $p$-Lipschitz function.
	\end{example}
	For some functions in the class  $\mathit{Lip}_{\alpha}$ (mainly for those in the class $\mathit{Lip}_{0}$) there have been some characterizations of dynamical properties, many of them in terms of the coefficients of the van der Put basis. For example, in \cite{Yu-Kh_JNT13} it is given the description of the $1$-Lipschitz functions which are measure preserving. Other characterizations of measure preserving, ergodic and locally scaling functions, are given in e.g. \cites{An94,An02,An,An-Kh_Book,An-Kh-Yu,Je-Li,Je_JNT,Je_PNUAA,Fur}. The class $\mathit{Lip}_\alpha$ is characterized in terms of the van der Put expansion \eqref{eq:vderp_univariate}, as follows.
	\begin{proposition}\protect{\cite{Je_JNT}*{Thm. 3.6}}
		If $f(x)=\sum_{m=0}^{\infty} B_me_m(x)$ is a continuous function from $\mathbb{Z}_p$ to $\mathbb{Z}_p$, then $f\in \mathit{Lip}_\alpha$ if and only if for every $m\geq 0$
		\begin{equation}\label{Eq:Coeff_of_p_alpha_Lipschitz}
			|B_m|_p\leq p^{-\lfloor \log_p m\rfloor+\alpha}.
		\end{equation}
	\end{proposition}
	In such case, one has $b_m^\alpha:=p^{-\lfloor \log_p m\rfloor+\alpha}B_m\in\mathbb{Z}_p$. 
	In particular, see \cite{An-Kh-Yu}*{Thm. 5}, for a continuous function from $\mathbb{Z}_p$ to $\mathbb{Z}_p$, $f(x)=\sum_{m=0}^{\infty} B_me_m(x)$, one has that $f\in \mathit{Lip}_0$ if and only if for every $m\geq 0$ one has $|B_m|_p\leq p^{-\lfloor \log_p m\rfloor}$. In this case $b_m:=p^{-\lfloor \log_p m\rfloor}B_m\in\mathbb{Z}_p$. 
	
	Another interesting feature of the class of $p^\alpha-$Lipschitz functions is its arithmetic nature, which is expressed for instance by the equivalence between \eqref{Eq:p_alpha_Lipschitz} and the following fact:  for $k\geq 1+\alpha$,
	\begin{equation}\label{Eq:p_alpha_equivalence}
		x\equiv y \bmod{p^k} \text{ implies } f(x)\equiv f(y) \bmod{p^{k-\alpha}}.
	\end{equation}
	This is the starting point for the following result.
	\subsection{Hensel's Lifting Lemma}
	Yurova and Khrennikov gave in \cite{Yu-Kh_JNT} a new criterion for lifting roots of $p$-adic continuous functions in the class $\mathit{Lip}_\alpha$, by looking at their van der Put expansion. A compelling fact about this result is that there is no assumption about differentiability (see Example \ref{Ex:Function} below), which is the usual requirement in several equivalent forms of Hensel's lemma, see e.g. \cites{Bor-Sha,Rib}. 
	
	The statement of their results are given in two separated cases, when $\alpha=0$ (see \cite{Yu-Kh_JNT}*{Thm. 2.1, Thm. 2.4}) and when $\alpha>0$ (see \cite{Yu-Kh_JNT}*{Thm. 3.2, Thm. 3.3}). In the latter case, the statement and the proof are given in terms of some sub-functions associated to $f\in\mathit{Lip}_\alpha$. We present below a uniform equivalent statement of these results, not involving sub-functions, but based on observation \eqref{Eq:Coeff_of_p_alpha_Lipschitz}. The proof of such statement is analogous to the proof of \cite{Yu-Kh_JNT}*{Thm. 2.4}, we present here our proof for the sake of completeness and also with the aim of motivate the proof of Theorem \ref{Thm:Main}. 
	
	For any $p$-adic integer $z$, we denote by $\overline{z}^{\mbox{\tiny k}}$ the reduction modulo $p^k$ of $z$. When $k=1$ we will just use $\overline{z}$. For a function $f: \mathbb{Z}_p\to\mathbb{Z}_p$,   $\overline{f}^{\mbox{\tiny k}}$ corresponds to the reduction modulo $p^k$ of all the coefficients of the van der Put expansion of $f$. 
	\begin{theorem}\label{Thm:Hensel_alpha_Univ}
		Let $f: \mathbb{Z}_p\to\mathbb{Z}_p$ be a function in the class $\mathit{Lip}_\alpha$, represented via van der Put series as $f(x)=\sum_{m=0}^{\infty}b_m^{\alpha}p^{\lfloor \log_p m\rfloor-\alpha}e_m(x).$
		\begin{enumerate}
			\item The function $f$ has a root in $\mathbb{Z}_p$ if and only if the equations $\overline{f}^{\mbox{\tiny k}}(x)\equiv 0 \bmod{p^{k-\alpha}},$ are solvable for every $k\geq 1+\alpha$. 
			\item Let $l_0$ be a positive integer and let $z$ be an integer with $0\leq z< p^{l_0+\alpha}$ such that \[f(z)\equiv 0 \bmod{p^{l_0+\alpha}}.\] For any non negative integers $l,m$ with $l_0+\alpha\leq l, m< p^{l}$ and such that $m\equiv z \bmod{p^{l_0+\alpha}}$, assume that 	
			\[\left\{p^{-\alpha}\overline{b^\alpha_{m+rp^l}}\ ;\ r=1,2,\ldots, p-1\right\}=\{1,2,\ldots, p-1\}.\]		
			Then there exists a unique  $\zeta\in \mathbb{Z}_p$ such that $f(\zeta)=0$ and $\zeta\equiv z \bmod{p^{l_0+\alpha}}$.
		\end{enumerate}
	\end{theorem}
	\begin{proof}
		The first part is an easy variation of the corresponding proof of \cite{Yu-Kh_JNT}*{Thm 2.1}. For the second part, 
		we shall show that it is possible to lift the root $z$ of $f$ modulo $p^{l_0+\alpha}$ to a root in $\mathbb{Z}_p$. We start by assuming that \[f(\hat{z}):=f(z+z_{l_0+\alpha+1}p^{l_0+\alpha+1}+\cdots+z_{l-1}p^{l-1})\equiv 0 \bmod{p^l},\] i.e. that $f(\hat{z})=p^l\cdot t$, for some $t\in\mathbb{Z}_p^\times$. Our first task is to find $r\in\{1,\ldots, p -1\}$ such that 
		\begin{equation}\label{Eq:ProofHensel1}
			f(\hat{z}+p^lr)\equiv 0 \bmod{p^{l+1}}.
		\end{equation}
		By Theorem \ref{Thm:vdP_base_univ}, $B_{\hat{z}+p^lr}=f(\hat{z}+p^lr)-f(\hat{z})$, thus \eqref{Eq:ProofHensel1} is reduced to $B_{\hat{z}+p^lr}+f(\hat{z})\equiv 0 \bmod{p^{l+1}}$, which in turn is reduced to $p^{l-\alpha}b_{\hat{z}+p^lr}^\alpha+tp^l\equiv 0 \bmod{p^{l+1}}$. Dividing by $p^{l}$, we get $\overline{p^{-\alpha}b_{\hat{z}+p^lr}^\alpha}+\bar{t}\equiv 0 \bmod{p}$, where we have emphasized the reduction $\mod p$ of the elements involved. This last equation has a unique solution for any $\bar{t}\in\{1,\ldots,p-1\}$ due to the hypothesis on the set $\left\{p^{-\alpha}\overline{b^\alpha_{m+rp^l}}\ ;\ r=1,2,\ldots, p-1\right\}$. Note that the solution of this last equation is equivalent to the solution of the congruence
		$b_{\hat{z}+p^lr}^\alpha+tp^{\alpha}\equiv 0 \bmod{p^{\alpha+1}}$.
		
		The process described above shows that one may construct a sequence \[Z=(z,z+z_{l_0+\alpha+1}p^{l_0+\alpha+1},\ldots, \hat{z},\hat{z}+z_lp^l,\ldots),\] where $f(\hat{z}+z_lp^l)\equiv 0\bmod{p^{l+1}}$ and $\hat{z}+z_lp^l\equiv \hat{z}\bmod{p^l}$. It follows that the sequence $Z$ converges to some $p$-adic integer $\zeta$ and $f(\zeta)=0$, being $\zeta$ unique by the continuity of $f$.
	\end{proof}
	\begin{example}\label{Ex:Function}
		Take $p=7$ and define $f:\ \mathbb{Z}_7\to \mathbb{Z}_7$ by $f(\sum_{i=0}^{\infty}x_{i}7^{i})=-5+\sum_{i=0}^{\infty}7^i(4+7i^3)x_{i}^5$. Then $f\in\mathit{Lip}_0$, is not differentiable at any point of $\mathbb{Z}_7$ but still one may lift the root $5$ to a $7$-adic root of $f$. Here $5$ is the solution to the equation $4x_0^5\equiv 5 \pmod{7}$. 
	\end{example}
	Some recent generalizations of Theorem \ref{Thm:Hensel_alpha_Univ} are given in \cites{Je_PNUAA,Yu-Kh_Izv,Ka-Sto}.
	
	\section{Elements of $p$-adic analysis in several variables}\label{Sec3}
	In this section we will develop some analogues of the results in Section \ref{Sec2} for multivariate functions, i.e. functions $F:\mathbb{Z}_p^n\to\mathbb{Z}_p$. 
	\subsection{$\mathbb{Q}_p^n$ and $p-$adic multivariate functions}
	We extend the $p$-adic norm to $\mathbb{Q}_{p}^{n}$ by taking
	\[||\boldsymbol{x}||_{p}:=\max_{1\leq i\leq n}|x_{i}|_{p},\quad\text{for}\quad\boldsymbol{x}=(x_{1},\dots,x_{n})\in\mathbb{Q}_{p}^{n}.\]
	We define $ord(\boldsymbol{x})=\min\limits_{1\leq i\leq n}\{ord(x_{i})\}$, then $||\boldsymbol{x}||_{p}=p^{-ord(\boldsymbol{x})}$. The metric space 
	$(\mathbb{Q}_{p}^{n},||\cdot||_{p})$ is a separable complete ultrametric space
	(here, separable means that $\mathbb{Q}_{p}^{n}$ contains a countable dense
	subset, which is $\mathbb{Q}^{n}$ ). For $r\in\mathbb{Z}$, we denote by
	\[B_{r}^{n}(\boldsymbol{a})=\{\boldsymbol{x}\in\mathbb{Q}_{p}^{n}\ :\ ||\boldsymbol{x}-\boldsymbol{a}||_{p}\leq p^{r}\}
	\]
	\textit{the ball of radius }$p^{r}$ \textit{with center at} $\boldsymbol{a}=(a_{1},\dots,a_{n})\in\mathbb{Q}_{p}^{n}$, and take $B_{r}^{n}(\boldsymbol{0}):=B_{r}^{n}$. Note that $B_{r}^{n}%
	(\boldsymbol{a})=B_{r}(a_{1})\times\cdots\times B_{r}(a_{n})$, where $B_{r}(a_{i}):=\{x_i\in\mathbb{Q}_{p}\ :\ |x_{i}-a_{i}|_{p}\leq p^{r}\}$ is the one-dimensional ball of radius
	$p^{r}$ with center at $a_{i}\in\mathbb{Q}_{p}$. The ball $B_{0}^{n}$ equals the product of $n$ copies of $B_{0}%
	=\mathbb{Z}_{p}$. We will prefer the notation $\mathbb{Z}_p^n$, since it is also the local ring of $\mathbb{Q}_{p}^{n}$. Note that, as in the one dimensional case, $\mathbb{Z}_p^n$ is a compact set in the topology of
	$(\mathbb{Q}_{p}^{n},||\cdot||_{p})$.
	
	If $\mathbb{S}\subseteq\mathbb{Q}_p^n$ then the concepts of continuity
	and derivative are defined in $\mathbb{S}$ in the same way as in $\mathbb{R}^n$. Again the set of locally constant functions provides a good example of continuous functions in several variables.
	\begin{definition}\label{Def:locContMult}
		A function $F:\ \mathbb{S} \to \mathbb{Q}_p$  is locally constant if for each $\boldsymbol{s}\in \mathbb{S}$ there exists a neighbourhood $U$ of $\boldsymbol{s}$ such that $F$ is constant on $U\cap \mathbb{S}$.
	\end{definition}
	For example, the product of $n$ univariate locally constant functions is a locally constant function of  $\mathbb{Q}_p^n$. If now we denote by 
	$C(\mathbb{S}\to \mathbb{Q}_p)$ the $\mathbb{Q}_p$-vector space of continuous functions from  $\mathbb{S}\subseteq\mathbb{Q}_p^n$ to $\mathbb{Q}_p$, then we may consider it as a non-Archimedean Banach space by defining the norm
	\[||F||_\infty=\sup_{\boldsymbol{x}\in \mathbb{S}}|F(\boldsymbol{x})|_p\quad\text{for}\quad F\in C(\mathbb{S}\to \mathbb{Q}_p).\]
	As in the one dimensional case, the set of locally constant functions on $\mathbb{S}$ form a $\mathbb{Q}_p$-linear subspace of $C(\mathbb{S}\to \mathbb{Q}_p)$ and moreover every continuous functions can be uniformly approximated by locally constant functions,
	more precisely.  
	\begin{proposition}
		Given $\epsilon>0$ and a function $F\in C(\mathbb{S}\to \mathbb{Q}_p)$ ($\mathbb{S}\subseteq\mathbb{Q}_p^n$), there exists a locally constant function $G:\ \mathbb{S} \to \mathbb{Q}_p$ such that $||F-G||_\infty<\epsilon$.  
	\end{proposition}
	\begin{proof}
		The proof of the one dimensional case, i.e. the one in  \cite{SchiBook}*{Thm. 26.2} can be easily adapted to the present situation.
	\end{proof}
	The van der Put Basis of  $C(\mathbb{Z}_p^n\to \mathbb{Q}_p)$, is formed by $n$-products of one dimensional locally constant functions, in fact, it is formed by products of $n$ elements of the set $\{e_m(x)\}_{m\in\mathbb{Z}_{\geq 0},x\in\mathbb{Z}_p}$ forming the one dimensional van der Put basis.
	\subsection{van der Put basis of $C(\mathbb{Z}_p^n\to \mathbb{Q}_p)$}
	In order to describe the orthonormal van der Put basis of $C(\mathbb{Z}_p^n\to \mathbb{Q}_p)$ we need some technical auxiliary functions. Take  $\boldsymbol{m}=(m_1,\ldots,m_n)\in\mathbb{Z}_{\geq 0}^n$ and define $I(\boldsymbol{m})=\{i_1,\ldots.i_k\}$ as the subset of $\{1,\ldots,n\}$ characterized by
	$i\in I(\boldsymbol{m})$ if and only if $m_i\geq p$. Note that for  $i\in I(\boldsymbol{m})$ it is well defined the operation $m_i^\ast$, while for  $m_{i_{k+1}},\ldots,m_{i_n}$ ( the entries of $\boldsymbol{m}$ which are less than $p$ ) the operation $m_i^\ast$ is not defined. Given a continuous function $F: \mathbb{Z}_p^n\to \mathbb{Q}_p$, we define recursively a family of functions $\{\phi_{1},\ldots,\phi_{k}\}$ as follows.  The function $\phi_{1}$ is defined by
	\[
	\begin{array}{cclc}
		\mathbb{Z}_{\geq 0}^{n-1}&\xrightarrow{\phi_{1}}&\mathbb{Q}_p\\
		(m_1,\ldots,\widehat{m}_{i_{1}},\ldots,m_{n})&\xrightarrow{\phantom{\phi_{1}}} &F(m_1,\ldots,m_{i_{1}},\ldots,m_n)-F(m_1,\ldots,m_{i_{1}}^\ast,\ldots,m_n).
	\end{array}
	\]
	The second function is defined as
	\begin{gather*}
		\phi_2(\ldots,\widehat{m}_{i_{2}},\ldots)=\phi_1(\ldots,\widehat{m}_{i_{1}},\ldots,m_{i_{2}},\ldots)-\phi_1(\ldots,\widehat{m}_{i_{1}},\ldots,m_{i_{2}}^\ast,\ldots)\\
		=F(\ldots,m_{i_{1}},\ldots,m_{i_{2}},\ldots)-F(\ldots,m_{i_{1}}^\ast,\ldots,m_{i_{2}},\ldots)-[F(\ldots,m_{i_{1}},\ldots,m_{i_{2}}^\ast,\ldots)\\
		-F(\ldots,m_{i_{1}}^\ast,\ldots,m_{i_{2}}^\ast,\ldots)].
	\end{gather*} 
	We continue in this fashion defining recursively the functions  $\phi_l$ for $2<l\leq k$, finally the function $\phi_{k}$ is defined by
	\begin{gather*}
		\phi_{k}(\ldots,\widehat{m}_{i_{k}},\ldots)=\phi_{k-1}(\ldots,\widehat{m}_{i_{1}},\ldots,\widehat{m}_{i_{k-1}},\ldots,m_{i_k},\ldots)\\
		-\phi_{k-1}(\ldots,\widehat{m}_{i_{1}},\ldots,\widehat{m}_{i_{k-1}},\ldots,m_{i_k}^\ast,\ldots).
	\end{gather*}

	\begin{theorem}\label{Thm:vdPut_multi}
		For $\boldsymbol{x}=(x_1,\ldots,x_n)\in\mathbb{Z}_p^n$ and $\boldsymbol{m}=(m_1,\ldots,m_n)\in\mathbb{Z}_{\geq 0}^n$, the functions 
		\[E_{\boldsymbol{m}}(\boldsymbol{x})= e_{m_1}(x_1)\cdots e_{m_n}(x_n),
		\]
		form an orthonormal basis (the van der Put basis) of the space $C(\mathbb{Z}_p^n\to \mathbb{Q}_p)$. Here the functions $e_{m_i}(x_i)$ are the functions described in Theorem \ref{Thm:vdP_base_univ}. 
		Moreover, if $F: \mathbb{Z}_p^n\to \mathbb{Q}_p$ is continuous and has the expansion
		\[F(\boldsymbol{x})=\sum_{\boldsymbol{m}\in \mathbb{Z}_{\geq 0}^n}A_{\boldsymbol{m}}E_{\boldsymbol{m}}(\boldsymbol{x}) =\sum_{m_1\geq 0}
		\cdots\sum_{m_n\geq 0} A_{m_{1},\ldots,m_{n}}\ e_{m_1}(x_1)\cdots e_{m_n}(x_n),
		\] 
		then $A_{m_1,\ldots,m_n}=F(m_1,\ldots,m_n),$ when $m_i<p$  for every $i=1,\ldots,n$. In other case and under the above assumptions 
		\begin{equation}\label{Eq:vdPut_thm2}
			A_{m_1,\ldots,m_n}=\phi_{k}(\ldots,\widehat{m}_{i_{k}},\ldots).
		\end{equation} 
		
	\end{theorem}
	\begin{proof}
		Note that the definition of  $e_{m_1}(x_1),\ \ldots,\ e_{m_n}(x_n)$ implies
		\[F(m_1,m_2,\ldots,m_n)=\sum_{i_1\vartriangleleft m_1}\sum_{i_2\vartriangleleft m_2}\cdots\sum_{i_n\vartriangleleft m_n} A_{i_{1},i_2,\ldots,i_{n}}.
		\]
		If every $m_i<p$, this implies $A_{m_1,\ldots,m_n}=F(m_1,\ldots,m_n)$. Otherwise, we assume that $I(\boldsymbol{m})=\{i_1,\ldots.i_k\}$, then 
		\[F(m_1,m_2,\ldots,m_n)=\sum_{j_{1}\vartriangleleft m_{i_{1}}}\sum_{j_{2}\vartriangleleft m_{i_{2}}}\cdots\sum_{j_{k}\vartriangleleft m_{i_k}}A_{(\boldsymbol{m},\boldsymbol{j})},\]
		where the vector $(\boldsymbol{m},\boldsymbol{j})$ has the entry $m_i$ for $i\notin I(\boldsymbol{m})$ and has the entry $j_{i}$ for $i\in I(\boldsymbol{m})$. 
		For ease of notation we will assume with out lost of generality that $I(\boldsymbol{m})=\{1,\ldots,k\}$, giving
		\[F(m_1,\ldots,m_k,m_{k+1},\ldots,m_n)=\sum_{j_{1}\vartriangleleft m_{1}}\sum_{j_{2}\vartriangleleft m_{2}}\cdots\sum_{j_{k}\vartriangleleft m_{k}}A_{j_{1},\ldots,j_k,m_{k+1},\ldots,m_{n}}.\]
		We next split the first sum up to $m_{1}^\ast$ and then $m_{1}$, to obtain
		\begin{gather*}
			F(m_1,\ldots,m_k,m_{k+1},\ldots,m_n)=\sum_{j_{1}\vartriangleleft m_{1}^\ast}\sum_{j_{2}\vartriangleleft m_{2}}\cdots\sum_{j_{k}\vartriangleleft m_{k}}A_{j_{1},\ldots,j_k,m_{k+1},\ldots,m_{n}}\\
			+\sum_{j_{2}\vartriangleleft m_{2}}\cdots\sum_{j_{k}\vartriangleleft m_{k}}A_{m_{1},j_2,\ldots,j_k,m_{k+1},\ldots,m_{n}}\\
			=F(m_1^\ast,\ldots,m_k,m_{k+1},\ldots,m_n)+\sum_{j_{2}\vartriangleleft m_{2}}\cdots\sum_{j_{k}\vartriangleleft m_{k}}A_{m_{1},j_2,\ldots,j_k,m_{k+1},\ldots,m_{n}},
		\end{gather*}
		which is equivalent to
		\begin{equation}\label{Eq:Proof_1vdP}
			\begin{gathered}
				F(m_1,\ldots,m_k,m_{k+1},\ldots,m_n)-F(m_1^\ast,\ldots,m_k,m_{k+1},\ldots,m_n)\\=\phi_{1}(\widehat{m}_{{1}},\ldots,m_{n})
				=\sum_{j_{2}\vartriangleleft m_{2}}\cdots\sum_{j_{k}\vartriangleleft m_{k}}A_{m_{1},j_2,\ldots,j_k,m_{k+1},\ldots,m_{n}}.
			\end{gathered}
		\end{equation}
		We now repeat the process and split the first sum in \eqref{Eq:Proof_1vdP} up to $m_{2}^\ast$ and then $m_{2}$, to obtain
		\begin{gather*}
			\phi_{1}(\widehat{m}_{{1}},m_2,\ldots,m_{n})
			=\sum_{j_{2}\vartriangleleft m_{2}^\ast}\cdots\sum_{j_{k}\vartriangleleft m_{k}}A_{m_{1},j_2,\ldots,j_k,m_{k+1},\ldots,m_{n}}\\
			+\sum_{j_{3}\vartriangleleft m_{3}}\cdots\sum_{j_{k}\vartriangleleft m_{k}}A_{m_{1},m_2,j_3,\ldots,j_k,m_{k+1},\ldots,m_{n}}.
		\end{gather*}
		By \eqref{Eq:Proof_1vdP} the first term in the RHS is equal to $\phi_{1}(\widehat{m}_{{1}},m_2^\ast,\ldots,m_{n})$, giving
		\begin{gather*}
			\phi_1(\widehat{m}_{{1}},m_2,\ldots,m_{n})-\phi_{1}(\widehat{m}_{{1}},m_2^\ast,\ldots,m_{n})
			=\phi_{2}(\widehat{m}_{{1}},\widehat{m}_2,m_3,\ldots,m_{n})\\
			=\sum_{j_{3}\vartriangleleft m_{3}}\cdots\sum_{j_{k}\vartriangleleft m_{k}}A_{m_{1},m_2,j_3,\ldots,j_k,m_{k+1},\ldots,m_{n}}.
		\end{gather*}
		After $k-1$ iterations of the process we get
		\begin{gather*}
			\phi_{k-1}(\widehat{m}_{{1}},\ldots,\widehat{m}_{{k-1}},m_k,\ldots,m_{n})=\sum_{j_k\vartriangleleft m_k} A_{m_{1},\ldots,m_{k-1},j_{k},m_{k+1},\ldots,m_n}\\
			=\sum_{j_k\vartriangleleft m_k^\ast} A_{m_{1},\ldots,m_{k-1},j_{k},m_{k+1},\ldots,m_n}+
			A_{m_{1},\ldots,m_{n}},
		\end{gather*}
		which proves \eqref{Eq:vdPut_thm2}. 
		
		We now proceed to show that $\{E_{\boldsymbol{m}}\}_{\boldsymbol{m}\in\mathbb{Z}_{\geq 0}^n}$ is an orthonormal base. We will consider $F(x_1,\ldots,x_n)$ an arbitrary element of $C(\mathbb{Z}_p^n\to \mathbb{Q}_p)$, and consider the series
		\begin{multline*}
			G(x_1,\ldots,x_n):=F(0,\ldots,0)\ e_0(x_1)\cdots e_0(x_n)\\+ \sum_{m_1,\ldots,m_n\in\mathbb{Z}_{\geq 0}} [\phi_{k}(\ldots,\widehat{m}_{i_{1}},\ldots,\widehat{m}_{i_{k}},\ldots)]\ e_{m_1}(x_1)\cdots e_{m_n}(x_n), 
		\end{multline*}	
		where there is at least one positive index. The uniform continuity of $F(x_1,\ldots,x_n)$ implies that 
		\[\lim_{m_1,\ldots,m_n\to \infty} \phi_{k}(\ldots,\widehat{m}_{i_{1}},\ldots,\widehat{m}_{i_{k}},\ldots)=0,\] 
		which implies that $G(x_1,\ldots,x_n)$ converges uniformly, i.e. $G\in C(\mathbb{Z}_p^n\to \mathbb{Q}_p)$. Now, from the definition of  $G$ we have 
		\[G(m_1,\ldots,m_n)=F(m_1,\ldots,m_n),\quad\text{for}\quad (m_1,\ldots,m_n)\in\mathbb{Z}_{\geq 0}^n,
		\]
		and then by continuity $G(x_1,\ldots,x_n)=F(x_1,\ldots,x_n)$. This shows that $\{E_{\boldsymbol{m}}\}_{\boldsymbol{m}\in\mathbb{Z}_{\geq 0}^n}$ is a generating set for $C(\mathbb{Z}_p^n\to \mathbb{Q}_p)$. Finally we shall show that $\{E_{\boldsymbol{m}}\}$ is an orthonormal set. To do so, note that, clearly $||F||_{\infty}\leq \sup_{\boldsymbol{m}}|A_{\boldsymbol{m}}|_p$. On the other hand, the ultrametric property implies that
		\begin{gather*}
			|A_{\boldsymbol{m}}|_p\leq \max\{|F(m_1,\ldots,m_n)|_p,|F(m_1,\ldots,m_{i_1}^\ast,\ldots,m_n)|_p,\ldots,\\
			|F(m_1,\ldots,m_{i_k}^\ast,\ldots,m_n)|_p,\ldots,|F(m_1,\ldots,m_{i_1}^\ast,\ldots,m_{i_2}^\ast,\ldots,m_n)|_p,\ldots,\\
			|F(m_1,\ldots,m_{i_1}^\ast,\ldots,m_{i_k}^\ast,\ldots,m_n)|_p\}\leq ||F||_{\infty}.
		\end{gather*}
		We conclude that $||F||_{\infty}=\sup_{\boldsymbol{m}}|A_{\boldsymbol{m}}|_p$, which completes the proof.
	\end{proof}
	\begin{remark}\label{Rmk1}
		\begin{enumerate}
			\item The first part of Theorem \ref{Thm:vdPut_multi} is proposed in \cite{SchiBook}*{Ex. 62D, pg. 192} as an exercise to the reader in the case of two variables. Our proof of this first part is inspired on the  proof of  \cite{SchiBook}*{Thm. 62.2}. The proof of the second part of Theorem \ref{Thm:vdPut_multi}, i.e. the determination of the coefficients, is our own contribution.
			\item Note that there are several equivalent choices for the auxiliary functions $\phi_l$. For instance we could have chosen as $\Phi_{1}$ the function
			\[	(m_1,\ldots,\widehat{m}_{i_{n}},\ldots,m_{n})\mapsto F(m_1,\ldots,m_{i_{n}},\ldots,m_n)-F(m_1,\ldots,m_{i_{n}}^\ast,\ldots,m_n),\]
			and then the subsequent $\Phi_l$ accordingly. The only change in our proof would have been the splitting step, beginning with the last sum each time.  
		\end{enumerate}
	\end{remark}
	
	\begin{example}\label{Ex:Coeffs}
		Let us illustrate the nested sequence of differences defining the coefficients of the van der Put expansion of Theorem \ref{Thm:vdPut_multi}. We take $n=4$ and 
		$(m_1,m_2,m_3,m_4)=(i,j,k,l)$, we further assume that $I(m_1,m_2,m_3,m_4)=\{i,j,k\}$. Then	
		\begin{gather*}
			A_{ijkl}=\phi_{3}(l)=F(i,j,k,l)-F(i^\ast,j,k,l)-[F(i,j^\ast,k,l)-F(i^\ast,j^\ast,k,l)]\\-[F(i,j,k^\ast,l)-F(i^\ast,j,k^\ast,l)]+[F(i,j^\ast,k^\ast,l)-F(i^\ast,j^\ast,k^\ast,l)].
		\end{gather*}
	\end{example}
	\subsection{Multivariate $p$-adic Lipschitz conditions}
	Differentiable functions over $\mathbb{Q}_p^n$ inherit some of the `bad' properties of their one dimensional counterpart, and here again one may consider for $\mathbb{S}\subseteq\mathbb{Q}_p^n$ the subspace of $C(\mathbb{S}\to \mathbb{Q}_p)$ formed by functions verifying a Lipschitz type condition. It seems that this type of conditions where first given for functions of several variables by Anashin in \cite{An94}.  
	\begin{definition}\protect{\cite{An-Kh_Book}*{Def. 3.25}}
		Take $\beta\in\mathbb{Z}_{\geq 0}$. A continuous function $F: \mathbb{Z}_p^n\to \mathbb{Z}_p$ is called a $p^\beta$-Lipschitz function if for every $\boldsymbol{x},\boldsymbol{y}\in\mathbb{Z}_p^n$,
		\[|F(\boldsymbol{x})-F(\boldsymbol{y})|_p\leq p^\beta||\boldsymbol{x}-\boldsymbol{y}||_p.\]
	\end{definition}
	When $\beta=0$, a $p^\beta$-Lipschitz function is just called a $1$-Lipschitz function. In fact, the study of continuous functions in several variables is only mentioned (to the best of our knowledge) in the works of Anashin in \cites{An94,An02,An-Kh_Book}. There, the author studies the class of functions that are differentiable modulo $p^k$, which attempts to approximate successive differentiability in each co-class modulo $p^k$. Then the author gives some characterizations of measure preserving and ergodic functions in terms of the coefficients of the Mahler base of functions $C(\mathbb{Z}_p^n\to \mathbb{Q}_p)$ under the $1$-Lipschitz condition. A natural question that we would like to address in a near future is to provide, or generalize, some of these characterizations but in terms of the van der Put coefficients and furthermore for a class of functions satisfying a more general Lipschitz condition that we next describe.
	
	Recall that for $\boldsymbol{x},\boldsymbol{y}\in\mathbb{Z}_p^n$, $||\boldsymbol{x}-\boldsymbol{y}||_{p}:=\max_{1\leq i\leq n}|x_{i}-y_i|_{p}$, then for a $p^\beta$-Lipschitz function we have 
	\[|F(\boldsymbol{x})-F(\boldsymbol{y})|_p\leq p^\beta\max_{1\leq i\leq n}|x_{i}-y_i|_{p}.\]
	This last inequality motivate the definition of the following Lipschitz condition. 
	\begin{definition}\label{Def:New_p_alpha}
		Take $\boldsymbol{\alpha}=(\alpha_1,\ldots,\alpha_n)\in\mathbb{Z}_{\geq 0}^n$. A function $F: \mathbb{Z}_p^n\to \mathbb{Z}_p$ is called a $p^{\boldsymbol{\alpha}}$-Lipschitz function if for every $\boldsymbol{x},\boldsymbol{y}\in\mathbb{Z}_p^n$,
		\[|F(\boldsymbol{x})-F(\boldsymbol{y})|_p\leq \max_{1\leq i\leq n}\{p^{\alpha_i}|x_i-y_i|_p\}.\]
	\end{definition} 
	
	\begin{example}
		\begin{enumerate}
			\item The function $F(x,y)=\frac{x-x^p}{p}+y$ is a $p^{(1,0)}$-Lipschitz function.
			\item For $i\in\{1,\ldots,n\}$, let $f_i(x_i)$ be a function in the class $\mathit{Lip}_{\alpha_i}$. Define $F(x_1,\ldots,x_n)$ as \[F(x_1,\ldots,x_n)=f_1(x_1)+\cdots+f_n(x_n).\] Then $F(x_1,\ldots,x_n)$ is a $p^{(\alpha_1,\ldots,\alpha_n)}$-Lipschitz function.
		\end{enumerate}
	\end{example}
	\begin{remark}
		Note that when $\boldsymbol{\alpha}=(\alpha_1,\ldots,\alpha_n)=(\beta,\ldots,\beta)$ with $\beta\in\mathbb{Z}_{\geq 0}$, a $p^{\boldsymbol{\alpha}}$-Lipschitz function is also a $p^\beta$-Lipschitz function. The Lipschitz condition of  Definition \ref{Def:New_p_alpha} can be naturally interpreted as a $p$-adic weighted Lipschitz condition with weight $\boldsymbol{\alpha}=(\alpha_1,\ldots,\alpha_n)\in\mathbb{Z}_{\geq0}^n$.
	\end{remark}
	The next result may be considered as a generalization of  \cite{An94}*{Prop. 1.4} for the case of $p^{\boldsymbol{\alpha}}$-Lipschitz functions, we need first some notation. For a function 
	\begin{center}
		\begin{tabular}{c c l}
			$F:$& $\mathbb{Z}_p^n$ & $\longrightarrow\ \mathbb{Z}_p$ \\
			& $\boldsymbol{x}=(x_1,\ldots,x_n)$ & $\longmapsto\ F(\boldsymbol{x}),$ \\
		\end{tabular}
	\end{center}
	and a fixed index $l\in\{1,\ldots,n\}$, we denote by $F_l(z)$ the projection function
	\begin{equation}\label{Eqn:Auxiliar_functions}
		\begin{tabular}{c c l}
			$F_l(z):$& $\mathbb{Z}_p$ & $\longrightarrow\ \mathbb{Z}_p$ \\
			& $z$ & $\longmapsto\ F_l(z)=F(x_1,\ldots,x_{l-1},z,x_{l+1},\ldots,x_n).$ 
		\end{tabular}
	\end{equation}
	
	\begin{proposition}\label{Prop:Mult_to_univ}
		Let $F:\mathbb{Z}_p^n\to\mathbb{Z}_p$ be a $p^{\boldsymbol{\alpha}}$-Lipschitz function, then the univariate function $F_l(z)\in\mathit{Lip}_{\alpha_l}$. Reciprocally, if $F:\mathbb{Z}_p^n\to\mathbb{Z}_p$ is a continuous function and if $F_l(z_l)\in\mathit{Lip}_{\alpha_l}$, for every $l\in\{1,\ldots,n\}$, then $F(\boldsymbol{x})$ is a $p^{(\alpha_1,\ldots,\alpha_n)}$-Lipschitz function.
	\end{proposition}
	\begin{proof}
		For the first part, assume that $z\equiv w \bmod{p^k}$, then $|z-w|_p\leq p^{-k}$ and $||(x_1,\ldots,x_{l-1},z,x_{l+1},\ldots,x_n)-(x_1,\ldots,x_{l-1},w,x_{l+1},\ldots,x_n)||_p\leq p^{-k}$. By the $p^{\boldsymbol{\alpha}}$-Lipschitz condition on $F$
		\[|F_l(z)-F_l(w)|_p\leq p^{\alpha_l-k},\] which shows that $F_l(z)\in\mathit{Lip}_{\alpha_l}$.
		
		For the second part we first consider the case $n=2$. Note that for every fixed $x_2\in\mathbb{Z}_p$, one has
		\begin{equation}\label{Eq:1Prop4.1}
			|F(z_1,x_2)-F(w_1,x_2)|_p=|F_1(z_1)-F_1(w_1)|_p\leq p^{\alpha_1}|z_1-w_1|_p,
		\end{equation}
		while for every fixed $x_1\in\mathbb{Z}_p$ 
		\begin{equation}\label{Eq:2Prop4.1}
			|F(x_1,z_2)-F(x_1,w_2)|_p=|F_2(z_2)-F_2(w_2)|_p\leq p^{\alpha_2}|z_2-w_2|_p.
		\end{equation}
		Replacing $x_2$ by $z_2$ in \eqref{Eq:1Prop4.1} and $x_1$ by $z_1$ in \eqref{Eq:2Prop4.1}, one gets
		\begin{gather*}
			|F(z_1,z_2)-F(w_1,w_2)|_p=|F(z_1,z_2)-F(w_1,z_2) +F(w_1,z_2)-F(w_1,w_2)|_p\\ 
			\leq \max\{|F(z_1,z_2)-F(w_1,z_2)|_p,|F(w_1,z_2)-F(w_1,w_2)|_p\}\\
			\leq \max\{p^{\alpha_1}|z_1-w_1|_p,p^{\alpha_2}|z_2-w_2|_p\}.
		\end{gather*}
		The general case follows by induction on $n$.
	\end{proof}
	In particular, one has that if $F$ is a $p^\beta$-Lipschitz function, each $F_l$ in \eqref{Eqn:Auxiliar_functions} belongs to $\mathit{Lip}_\beta$. And if each projection $F_l$ belongs to $\mathit{Lip}_\beta$, then $F$ is a $p^\beta$-Lipschitz function.
	
	Now, for a given $\boldsymbol{\alpha}=(\alpha_1,\ldots,\alpha_n)\in\mathbb{Z}_{\geq 0}^n$, let $N_{\boldsymbol{x},\boldsymbol{y}}(\boldsymbol{\alpha}):=N\in\{1,2,\ldots,n\}$ be the index such that $p^{\alpha_{N}}|x_{N}-y_{N}|_p=\max_{1\leq i\leq n}\{p^{\alpha_i}|x_i-y_i|_p\}$, then the analogue of property \eqref{Eq:p_alpha_equivalence}, can be stated as follows. $F:\mathbb{Z}_p^n\to\mathbb{Z}_p$  is a $p^{\boldsymbol{\alpha}}$-Lipschitz function if and only if
	\begin{equation*}
		x_{N}\equiv y_{N} \bmod{p^k} \text{ implies } F(\boldsymbol{x})\equiv F(\boldsymbol{y}) \bmod{p^{k-\alpha_{N}}}, \text{ for } k\geq 1+\alpha_{N}.
	\end{equation*}
	It is also possible in the $n$-dimensional case to give an estimation of the coefficients of $p^{\boldsymbol{\alpha}}$-Lipschitz functions in terms of van der Put expansions.
	\begin{theorem}\label{Thm:alpha_lip_to_vdPut_multivar}
		Let $F(\boldsymbol{x})=\sum_{\boldsymbol{m}\in \mathbb{Z}_{\geq 0}^n}A_{\boldsymbol{m}}E_{\boldsymbol{m}}(\boldsymbol{x})$ be a continuous function from $\mathbb{Z}_p^n$ to $\mathbb{Z}_p$. If $F$ is a $p^{\boldsymbol{\alpha}}$-Lipschitz function then 
		\[|A_{m_{1},\ldots,m_{n}}|_p\leq p^{\min\{-\lfloor \log_p m_{i_1}\rfloor+\alpha_{i_1},\ldots, -\lfloor \log_p m_{i_k}\rfloor+\alpha_{i_k}\}}=p^{\min_{i\in I(\boldsymbol{m})}\{-\lfloor \log_p m_i\rfloor+\alpha_i\}},\] for every $\boldsymbol{m}=(m_{1},\cdots,m_{n})\in \mathbb{Z}_{\geq 0}^n$.
	\end{theorem}
	\begin{proof}
		By the definition of $A_{m_{1},\ldots,m_{n}}$ and the fact that $F_{i_k}\in \mathit{Lip}_{\alpha_{i_k}}$ we have
		\begin{gather*}
			|A_{m_{1},\ldots,m_{n}}|_p\\
			=|\phi_{k-1}(\ldots,\widehat{m}_{i_{1}},\ldots,\widehat{m}_{i_{k-1}},\ldots,m_{i_k},\ldots)
			-\phi_{k-1}(\ldots,\widehat{m}_{i_{1}},\ldots,\widehat{m}_{i_{k-1}},\ldots,m_{i_k}^\ast,\ldots){\tiny }|_p\\
			\leq p^{-\lfloor \log_p m_{i_k}\rfloor+\alpha_{i_k}}.
		\end{gather*}
		Now, from the second part of Remark \ref{Rmk1} we know that we may choose another equivalent set of auxiliary functions $\Phi_l$. In particular if we stick to the choice given there, we would arrive to an equality of the form
		\begin{gather*}
			A_{m_{1},\ldots,m_{n}}=\Phi_{k}(\ldots,\widehat{m}_{i_{1}},\ldots,\widehat{m}_{i_{k}},\ldots)\\
			=\Phi_{k-1}(\ldots,m_{i_1},\ldots,\widehat{m}_{i_{2}},\ldots,\widehat{m}_{i_{k}},\ldots)
			-\Phi_{k-1}(\ldots,m_{i_1}^\ast,\ldots,\widehat{m}_{i_{2}},\ldots,\widehat{m}_{i_{k}},\ldots),
		\end{gather*}
		from which $|A_{m_{1},\ldots,m_{n}}|_p\leq p^{-\lfloor \log_p m_{i_{1}}\rfloor+\alpha_{i_{1}}}$. In general we will have
		\[|A_{m_{1},\ldots,m_{n}}|\leq p^{\min\{-\lfloor \log_p m_{i_1}\rfloor+\alpha_{i_1},\ldots, -\lfloor \log_p m_{i_k}\rfloor+\alpha_{i_k}\}}.\]
	\end{proof}
	\begin{corollary}\label{Coro}
		When $F(\boldsymbol{x})=\sum_{\boldsymbol{m}\in \mathbb{Z}_{\geq 0}^n}A_{\boldsymbol{m}}E_{\boldsymbol{m}}(\boldsymbol{x})$, is a $p^{\boldsymbol{\alpha}}$-Lipschitz function, then 
		\[a_{m_{1},\ldots,m_{n}}:=p^{\min\{-\lfloor \log_p m_{i_1}\rfloor+\alpha_{i_1},\ldots, -\lfloor \log_p m_{i_k}\rfloor+\alpha_{i_k}\}}A_{m_{1},\ldots,m_{n}}\in\mathbb{Z}_p.\]
		Equivalently 
		\begin{align*}
			A_{m_{1},\ldots,m_{n}}&= p^{-\min\{-\lfloor \log_p m_{i_1}\rfloor+\alpha_{i_1},\ldots, -\lfloor \log_p m_{i_k}\rfloor+\alpha_{i_k}\}}a_{m_{1},\ldots,m_{n}}\\
			&=p^{\max\{\lfloor \log_p m_{i_1}\rfloor-\alpha_{i_1},\ldots, \lfloor \log_p m_{i_k}\rfloor-\alpha_{i_k}\}}a_{m_{1},\ldots,m_{n}},
		\end{align*}
		for some $a_{m_{1},\ldots,m_{n}}\in\mathbb{Z}_p$.
	\end{corollary}
	In particular, when $F$ is a $p^\beta$-Lipschitz function then for every $\boldsymbol{m}=(m_{1},\cdots,m_{n})\in \mathbb{Z}_{\geq 0}^n$
	\[|A_{m_{1},\ldots,m_{n}}|\leq p^{\min\{-\lfloor \log_p m_1\rfloor,\ldots, -\lfloor \log_p m_n\rfloor\}+\beta}
	\]
	Moreover, when $F$ is a $1$-Lipschitz function, one may assume that 
	\[A_{m_{1},\ldots,m_{n}}=p^{\max\{\lfloor \log_p m_1\rfloor,\ldots, \lfloor \log_p m_n\rfloor\}}a_{m_{1},\ldots,m_{n}}\,
	\]
	for some $a_{m_{1},\ldots,m_{n}}\in\mathbb{Z}_p$. 
	
	\subsection{Multivariate Hensel's Lemma}
	Finally we present a version of Hensel's lifting Lemma for functions of several variables, generalizing thus Theorem \ref{Thm:Hensel_alpha_Univ}.
	\begin{theorem}\label{Thm:Main}
		Let $F: \mathbb{Z}_p^n\to\mathbb{Z}_p$ be a $p^{\boldsymbol{\alpha}}$-Lipschitz function, represented via van der Put series as 
		\[F(\boldsymbol{x})=\sum_{m_1\geq 0}
		\cdots\sum_{m_n\geq 0} p^{\max\{\lfloor \log_p m_1\rfloor -\alpha_1,\ldots, \lfloor \log_p m_n\rfloor -\alpha_n\}}a_{m_{1},\ldots,m_{n}}E_{m_1,\ldots,m_n}(\boldsymbol{x}).
		\]
		\begin{enumerate}
			\item The function $F$ has a root in $\mathbb{Z}_p^n$ if and only if there exist at least one index $j\in\{1,\ldots,n\}$, such that the projection function $F_j(z)$ (defined in \ref{Eqn:Auxiliar_functions}) has a root.
			\item Let $l_0$ be a positive integer and let $\boldsymbol{z}=(z_1,\ldots,z_n)\in\mathbb{Z}^n$ with $0\leq z_k< p^{l_0+\alpha_k}$ for $k=1,\ldots,n$ and satisfying 
			\[F(\boldsymbol{z})\equiv 0 \bmod{p^{l_0+\min\{\alpha_1,\ldots,\alpha_n\}}}.\] 
			Consider a non negative integer $l$ with $l\geq l_0+\max\{\alpha_1,\ldots,\alpha_n\}$. Set also $\boldsymbol{m}=(m_{1},\ldots,m_{n})\in \mathbb{Z}_{\geq 0}^n$ satisfying $m_i< p^{l}$ for $i=1,\ldots,n$ and $m_i\equiv z_i \bmod{p^{l_0+\alpha_i}}$. For  $l$ and $\boldsymbol{m}$ as above assume that there exists at least one index $j$ for which
			\begin{gather*}
				\left\{p^{-\alpha_j}\overline{\phi_{1}(m_1,\ldots,m_{j-1},\widehat{m_j+r p^l},m_{j+1},\ldots,m_{n})}\ ;\ r=1,2,\ldots, p-1\right\}\\
				=\{1,2,\ldots, p-1\}.
			\end{gather*} 
			Then there exists a unique  $\boldsymbol{\zeta}\in \mathbb{Z}_p$ such that $F(\boldsymbol{\zeta})=0$ and $\zeta_k\equiv z_k \bmod{p^{l_0+\alpha_k}}$, for every $k\in\{1,\ldots,n\}$.
		\end{enumerate}
	\end{theorem}
	
	\begin{proof}
		The proof of the first part is an easy variation of the proof of \cite{Yu-Kh_JNT}*{Theorem 2.1}, taking into account Proposition \ref{Prop:Mult_to_univ}. For the proof of the second part we will lift the root $\boldsymbol{z}\in\mathbb{Z}_{\geq 0}$  to a root $\boldsymbol{\zeta}=(\zeta_1,\ldots,\zeta_n)\in\mathbb{Z}_p^n$ of the function $F$. Assume that $\boldsymbol{m}=(m_1,\ldots,m_n)$ is a root of $F$ modulo $p^l$ such that for every $i=1,\ldots,n$ one has $m_i<p^l$ and moreover   
		$m_i\equiv z_i \bmod{p^{l_0+\alpha_i}}$. 
		
		We will follow the idea of \cite{Bor-Sha}*{Thm. 3, Sec. 5.2}, to show that	there exists some $r\in\{1,\ldots, p -1\}$ with
		\begin{equation*}
			F(m_1,\ldots,m_{j-1},m_{j}+r\cdot p^l,m_{j+1},\ldots,m_n)\equiv 0 \bmod{p^{l+1}}.
		\end{equation*}
		Without loss of generality we will assume that $j=i_1=1$. Now, from the proof of Theorem \ref{Thm:vdPut_multi} we know that
		\begin{gather*}
			F(m_1+r p^l,m_2,\ldots,m_n)=F(m_1,m_2,\ldots,m_n)+\sum_{j_{2}\vartriangleleft m_{2}}\cdots\sum_{j_{n}\vartriangleleft m_{n}}A_{m_1+r p^l,j_2,\ldots,j_{n}}\\
			=F(m_1,m_2,\ldots,m_n)+\phi_{1}(\widehat{m_1+r p^l},m_2,\ldots,m_{n}),
		\end{gather*}
		which implies that we want  $r\in\{1,\ldots, p -1\}$ verifying
		\begin{gather}\label{Eq:ProofMainThm2}
			F(m_1,m_2,\ldots,m_n)+\phi_{1}(\widehat{m_1+r p^l},m_2,\ldots,m_{n})\equiv 0 \bmod{p^{l+1}}.
		\end{gather}
		By hypothesis $F(\boldsymbol{m})\equiv 0 \bmod{p^{l}}$, i.e. $F(\boldsymbol{m})=p^lt$, for some $t\in\mathbb{Z}_p^\times$. On the other hand one has by construction that  $p^{l-\alpha_1}$ divides $\phi_{1}(\widehat{m_1+r p^l},m_2,\ldots,m_{n})$, using Proposition \ref{Prop:Mult_to_univ} and the fact that $F_1\in\mathit{Lip}_{\alpha_1}$.  Dividing by $p^{l}$ in \eqref{Eq:ProofMainThm2}, we reduce our task to find $r\in\{1,\ldots, p -1\}$ such that
		\[\bar{t}+p^{-\alpha_1}\overline{\phi_{1}(\widehat{m_1+r p^l},m_2,\ldots,m_{n})}\equiv 0 \bmod{p}.
		\]
		This is precisely the case under the hypothesis on the set 	
		\[\left\{p^{-\alpha_j}\overline{\phi_{1}(m_1,\ldots,m_{j-1},\widehat{m_j+r p^l},m_{j+1},\ldots,m_{n})}\ ;\ r=1,2,\ldots, p-1\right\}.
		\]
		
		Then we may repeat the same steps in order to lift the root  $(m_1+r p^l,m_2,\ldots,m_n)$ to a root $\boldsymbol{\mu}$ of F modulo $p^{l+2}$ with the property
		that $\mu_i\equiv m_i \bmod{p^{l}}$ for every $i=1,\ldots,n$. The sequence of liftings obtained this way converges to some  $\boldsymbol{\zeta}=(\zeta_1,\ldots,\zeta_n)\in\mathbb{Z}_p^n$ for which $F(\boldsymbol{\zeta})=0$.
	\end{proof}
	It would be nice to have a statement of Theorem \ref{Thm:Main} involving only the coefficients of the van der Put series of $f$.
	\bibliographystyle{plain}

\end{document}